\documentclass[12pt]{amsart}

\usepackage{amsmath,amssymb,verbatim}
\usepackage{amsthm}
\usepackage{stmaryrd}
\usepackage{enumerate}
\usepackage{url}
\usepackage{wasysym}
\usepackage[letterpaper, portrait, margin=1in]{geometry}

\newtheorem*{theorem*}{Theorem}
\newtheorem{theorem}{Theorem}[section]
\newtheorem{fact}[theorem]{Fact}
\newtheorem{lemma}[theorem]{Lemma}
\newtheorem{corollary}[theorem]{Corollary}
\newtheorem{proposition}[theorem]{Proposition}

\theoremstyle{definition}
\newtheorem{definition}[theorem]{Definition}

\newtheorem{remark}[theorem]{Remark}

\def\Def{\operatorname{Def}}
\def\Aut{\operatorname{Aut}}
\def\Av{\operatorname{Av}}
\def\Eval{\operatorname{Eval}}
\def\conv{\operatorname{conv}}

\title{Local Keisler Measures and NIP Formulas}
\author[K. Gannon]{K. Gannon}
\address{Department of Mathematics\\
University of Notre Dame\\
Notre Dame, IN, 46656\\
 USA}
\email{kgannon1@nd.edu}
\date{November 5, 2018}
\thanks{This material is based upon research supported by the Chateaubriand Fellowship of the Office for Science \&
Technology of the Embassy of France in the United States. This material was also supported in part by NSF Grant DMS-1500671.}

\begin{document}
\begin{abstract} We study generically stable measures in the local, NIP context. We show that in this setting, a measure is generically stable if and only if it admits a natural finite approximation. 
\end{abstract}
\maketitle
\section{Introduction} 
The connection between finitely additive probability measures and NIP theories was first noticed by Keisler in his seminal paper \cite{K}. Around 20 years later, the work of Hrushovski, Peterzil, Pillay, and Simon greatly expanded this connection in \cite{NIP1}, \cite{NIP2}, and \cite{NIP3}. In particular, they introduce the notion of \textit{generically stable measures}. These measures exhibit properties of measures in stable theories. One of the most important properties of generically stable measures is that they admit a particular kind of approximation. The following is weaker version of the approximation theorem proven in \cite{NIP3},

\begin{theorem*}[\cite{NIP3}]\label{Original} Assume that $T$ is an NIP theory. Let $\mathcal{U}$ be a monster model of $T$, $M$ be a small elementary substructure of $\mathcal{U}$, and $\mu$ be a finitely additive measure on the Boolean algebra of definable sets with parameters from $\mathcal{U}$ (in free variables $\overline{x}$). Then, $\mu$ is generically stable over $M$ if and only if $\mu$ is finitely approximable over $M$, i.e. for every partitioned formula $\varphi(\overline{x};\overline{y})$, and for every $\epsilon>0$, there exists a sequence of points $a_1,...,a_n$ in $M^{\overline{x}}$ such that for any $\overline{c} \in \mathcal{U}^{\overline{y}}$,
\begin{equation*}
    |\mu(\varphi(\overline{x};\overline{c})) - Av(\overline{a})(\varphi(\overline{x};\overline{c}))| < \epsilon. 
\end{equation*}
\end{theorem*}

The purpose of this paper is to prove a local version of the theorem above. We prove the following, 

\begin{theorem*}[Main Theorem]\label{Main Theorem} Let $T$ be a first order theory. Let $\mathcal{U}$ be a monster model of $T$, $M$ be a small elementary substructure of $\mathcal{U}$, and $\varphi(\overline{x};\overline{y})$ an NIP formula. Let $\mu$ be a finitely additive measure on the collection of $\varphi$-definable sets with parameters from $\mathcal{U}$. Then $\mu$ is generically stable over $M$ (as in Definition \ref{cont}) if and only if $\mu$ is finitely approximable over $M$ (as in Definition \ref{finap}). 
\end{theorem*}

The proof of our Main Theorem involves translating the concept of \textit{generic stability} into the framework of functional analysis. From this vantage point, we can apply an important result of Bourgain, Fremlin, and Talagrand \cite{BFT}, namely Theorem \ref{equiv}. The connection between NIP formulas, local types, and Theorem \ref{equiv} was first noticed by Simon in \cite{S} and further work has been done by Khanaki and Pillay in \cite{Kh} and \cite{KP}. We can further extend this connection to local measures via Ben Yaacov's work on continuous VC classes \cite{B}. 

This paper is organized as follows: In section 2, we fix some model theoretic conventions. We will also provide all necessary functional analysis background. 

In section 3, we connect NIP formulas, continuous VC classes, and the important result of Bourgain, Fremlin, and Talagrand mentioned earlier in the introduction. 

In section 4, we begin by defining some important properties for local measures. We then translate these properties into the language of functional analysis. Using the theorems from functional analysis outlined in section 2 and the connection established in section 3, we will prove our Main Theorem. 

\subsection*{Acknowledgments} I would first and foremost like to thank my advisor Sergei Starchenko for many helpful and indispensable discussions. I would also like to thank Elisabeth Bouscaren, Ita\"i Ben Yaacov, and Artem Chernikov for fruitful discussions during my time in Paris and the Chateaubraind Fellowship for making this travel possible. Finally, I thank Gabriel Conant and Jinhe Ye for many editorial suggestions and corrections. 

\section{Notation and Preliminaries} 
\subsection{Model Theory Conventions} Our notation for the most part is standard.  We refer the reader to \cite{guide} for a thorough introduction to NIP theories and formulas.

Throughout this paper we fix $T$ a first order theory in a countable\footnote{For uncountable theories, we simply take the reduct to a countable language containing all the formulas we are working with.} language $L$, $\mathcal{U}$ a monster model of $T$, and $M$ a small elementary substructure of $\mathcal{U}$. The letters $x$ and $y$ will denote tuples of variables. We let $l(x)$ denote the length of a tuple. We let $\varphi(x;y)$ be a partitioned formula (without parameters) and with object variables $x$ and parameters variables $y$. The formula $\varphi^*(x;y)$ will denote the exact same formula as $\varphi(x;y)$, but with the roles exchanged for parameter and variable tuples. 

We recall that a formula $\varphi(x;y)$ is IP if for every natural number $n$, there exists $A \subset \mathcal{U}^{l(x)}$ such that $|A| \geq n$ for every $K \subseteq A$ there exists some $b_{K}$ in $\mathcal{U}^{l(y)}$, so that $\{c \in A: \mathcal{U} \models \varphi(c,b_{K})\} = K$. If $\varphi(x;y)$ is not IP, then we say that $\varphi(x;y)$ is NIP. We now recall two basic facts whose proofs can be found in \cite{guide}. First, if $\varphi(x;y)$ is NIP, then $\varphi^{*}(x;y)$ is NIP. Second, any Boolean combination of NIP formulas is also NIP. 

If $B \subseteq \mathcal{U}$, we let $S_x(B)$ denote the space of types in variable $x$ with parameters from $B$. Moreover, we denote the space of $\varphi$-types in object variable $x$ with parameters from $B$ as $S_{\varphi}(B)$. If $b \in \mathcal{U}^n$, we let $tp(b/B)$ denote the type of $b$ over $B$. We will routinely write $\mathcal{U}^{l{(x)}}$ as $\mathcal{U}^x$ or even simply as $\mathcal{U}$ when $l(x)$ is clear. We let $\Def_{x}(\mathcal{U})$ be the Boolean algebra of all definable subsets of $\mathcal{U}^{x}$ and let $\Def_{\varphi}(\mathcal{U})$ be the subalgebra of $\Def_{x}(\mathcal{U})$ generated by $\{\varphi(x;b): b \in \mathcal{U}^{y}\}$. 

A \textit{formula} is an $L$-formula and \textit{$\varphi$-formula} is a Boolean combination of elements in $\{\varphi(x;b): b \in \mathcal{U}^{y}\}$. We will abuse notation and routinely identify definable sets with the formulas which define them. We let $\mathfrak{M}_x({\mathcal{U}})$ be the collection of finitely additive probability measures on $\Def_{x}(\mathcal{U})$ and $\mathfrak{M}_{\varphi}(\mathcal{U})$ be the collection of finitely additive probability measures on $\Def_{\varphi}(\mathcal{U})$.

If $\overline{a} = a_1,...,a_n$ is a sequence of points in $\mathcal{U}^{x}$, we let $\Av(\overline{a})$ be the measure on $\Def_{x}(\mathcal{U})$ where for every formula, $\psi(x)$, 
\begin{equation*}
    \Av(\overline{a})(\psi(x)) = \frac{|\{i: \mathcal{U} \models \psi(a_i)\}|}{n}.
\end{equation*} 
Finally, we let $\Aut(\mathcal{U}/M)$ denote the collection of automorphisms of $\mathcal{U}$ which fix $M$ pointwise. 

\subsection{Functional Analysis} We recall some definitions and theorems from functional analysis. We refer the reader to \cite{Conway} as a standard reference on the subject. 

Let $X$ be a set, and let $\mathbb{R}^X$ denote the collection of all functions from $X$ to $\mathbb{R}$. $\mathbb{R}^X$ is a topological space with the product topology. If $A \subseteq \mathbb{R}^X$, we let $cl(A)$ be the topological closure of $A$ in $\mathbb{R}^X$. 

Let $(f_i)_{i \in \mathbb{N}}$ be a sequence in $\mathbb{R}^X$. We recall two different notions of convergence:
\begin{enumerate}
    \item $(f_i)_{i \in \mathbb{N}}$ \textit{converges pointwise} to a function $f$, written $f_i \to f$, if for every $b \in X$ and for every $\epsilon >0$ there is some natural number $N$ where for any $n > N$, $|f_{n}(b) - f(b)| < \epsilon$. 
    \item $(f_i)_{i \in \mathbb{N}}$  \textit{converges uniformly} to a function $f$, written $f_i \to_{u} f$, if for every $\epsilon >0$ there is some natural number $N$ where for any $n > N$, $\sup_{b \in X}|f_{n}(b) - f(b)| < \epsilon$.
\end{enumerate}

Assume that $X$ is a topological space and let $C(X)$ denote the space of continuous from $X$ to $\mathbb{R}$. If $X$ is a compact Hausdorff space, then $C(X)$ is a Banach space with the uniform norm, $||\cdot||_{\infty}$, where $||f||_{\infty} = \sup_{x \in X} |f(x)|$.

Let $Y$ be a real Banach space, let $y \in Y$, and let $(y_i)_{i \in \mathbb{N}}$ be a sequence in $Y$. We say that $(y_i)_{i \in \mathbb{N}}$ \textit{converges weakly} to $y$, written $y_i \to_w y$, if for all continuous linear functionals $G:Y \to \mathbb{R}$, we have that $\lim_{i \in \mathbb{N}} G(y_i) = G(y)$.

We note if $X$ is a compact Hausdorff space, such as $S_{y}(M)$, then for any sequence of functions $(f_i)_{i \in \mathbb{N}}$ in $C(X)$, one may determine whether this sequence converges pointwise, weakly, or uniformly. This is clear since $C(X)$ is a Banach space and $C(X) \subset \mathbb{R}^X$. After the following definition, we will state some theorems which connect these notions of convergence.

\begin{definition}[Convex Combination] Let $Y$ be a vector space and $A \subseteq Y$. We let $\conv(A)$ be the convex hull of $A$, and we let $\conv_{\mathbb{Q}}(A)$ denote the collection of all rational convex combinations of elements from $A$, i.e.,
\begin{equation*} 
\conv_{\mathbb{Q}}(A) = \left\{\sum^n_{i=1}r_ia_i : a_i \in A \text{ ; } n \in \mathbb{N} \text{ ; }  r_i \in \mathbb{Q}^+ \text{ ; }  \sum_{i=1}^nr_i=1\right\}.
\end{equation*}
\end{definition}

We now recall some theorems from functional analysis. Our first theorem is a trivial consequence of Mazur's Lemma \cite{BS} and connects the notions of uniform convergence and weak convergence. We will refer to the following theorem simply as Mazur's Lemma. 

\begin{lemma}[Mazur's Lemma \cite{BS}]\label{GR} Let $(Y,||\cdot||)$ be a Banach space, $y \in Y$, $(a_i)_{i \in \mathbb{N}}$ be a sequence of elements in $Y$, and $A = \{a_i: i \in \mathbb{N}\}$. If $a_i \to_w y$, then there is a sequence of $z_i \in \conv_{\mathbb{Q}}(A)$ such that $\lim_{i\to \infty}||z_i - y|| = 0$.
\end{lemma} 

In particular, if $X$ is a compact Hausdorff space, $(f_i)_{i \in \mathbb{N}}$ is a sequence in $C(X)$, and $f_i \to_{w} g$, then there exists $h_i \in \conv_{\mathbb{Q}}(\{f_i: i \leq \mathbb{N}\})$ such that $h_i \to_u g$. 

The following theorem is an application of the the dominated convergence theorem and the Riesz representation theorem. This theorem connects the notions of pointwise convergence and weak convergence. 

\begin{theorem}[\cite{Semadeni}]\label{K} Let $K$ be a compact Hausdorff space, $f \in C(K)$, and $(f_i)_{i \in \mathbb{N}}$ be a sequence in $C(K)$. Then the following are equivalent: 
\begin{enumerate}[($i$)]
\item $f_i \to_w f$.
\item $f_i \to f$ and $\sup_{i \in \mathbb{N}}||f_i||_{\infty}< \infty$. 
\end{enumerate} 
\end{theorem} 

The next theorem is the result by Bourgain, Fremlin, and Talagrand \cite{BFT} which we alluded to in the introduction. In fact, a much more general statement is proven in \cite{BFT} than the one we provide. The connection between this theorem and NIP formulas is well known and has been expanded upon in \cite{S}, \cite{Kh}, and \cite{KP}. This theorem is slightly more technical than the last two, but essentially it allows one to find pointwise convergent sequences under particular tameness conditions. Before we can state the theorem, we define the tameness condition. 

\begin{definition}[Sequential Independence] Let $\mathbf{X} \subset \mathbb{R}^Y$. Then we say that $\mathbf{X}$ is \textit{sequentially independent} if there exists a sequence $(f_n)_{n \in \mathbb{N}}$ of elements in $\mathbf{X}$, an $r \in \mathbb{R}$, and an $\epsilon > 0$ such that for every $I \subseteq \mathbb{N}$, there exists some  $b_{I}$ in $Y$ such that, 
\begin{equation*}
    \{n \in \mathbb{N}: f_n(b_{I}) \leq r\} = I \text{ and } \{n \in \mathbb{N}: f_n(b_{I}) \geq r + \epsilon \} = \mathbb{N}\backslash I. 
\end{equation*}
If $\mathbf{X}$ is not sequentially independent, we say that $\mathbf{X}$ is \textit{sequentially dependent}. 
\end{definition}

\begin{theorem}[\cite{BFT}]\label{equiv} Let $X$ be a compact Hausdorff space, $A \subseteq C(X)$, and $|A| \leq \aleph_0$.  Assume that $\sup_{f \in A} ||f||_{\infty} < \infty$. Then the following are equivalent: 
\begin{enumerate}[($i$)]
\item $A$ is sequentially dependent. 
\item for each $f \in \mathbb{R}^X$, if $f \in cl(A)$, then there exists a sequence of elements $(f_i)_{ i \in \mathbb{N}}$ from $A$ such that $f_i \to f$.
\end{enumerate} 
\end{theorem}

\subsection{Model Theory and Functions} In this subsection, we briefly discuss special types of functions. We recall that $\mathcal{U}$ is a fixed monster model of $T$ and $M$ is a small elementary substructure of $\mathcal{U}$.

\begin{definition}\label{F}
If $\varphi(x;y)$ is a partitioned formula and $a$ is in $M^{x}$, we define the map $F_{a}^{\varphi}: S_{y}(M)\to \{0,1\} \subset \mathbb{R}$ via,
\begin{equation*}
    F_{a}^{\varphi}(p)=\begin{cases}
\begin{array}{cc}
1 & \varphi(a,y)\in p,\\
0 & \text{otherwise.}
\end{array}\end{cases}
\end{equation*}
Moreover, we let $\mathbb{F}_M^{\varphi} = \{F_{a}^{\varphi} : a \in M\}$. 
\end{definition}

\begin{remark} It is clear that for each $a$ in $M^{x}$ the map $F^{\varphi}_{a}$ is continuous. We will always view $F_{a}^{\varphi}$ as a map from $S_{y}(M)$ to $\mathbb{R}$. We let $\conv(\mathbb{F}^{\varphi}_{M})$ be the convex hull of $\mathbb{F}^{\varphi}_M$ in $C(S_{y}(M))$. 
\end{remark}

\begin{definition}[Representative Sequence] Let $f \in \conv_{\mathbb{Q}}(\mathbb{F}_{M}^{\varphi})$. A \textit{representative sequence} for $f$ is a sequence of points $a_1,...,a_m \in M^{x}$  so that for every $b \in \mathcal{U}^y$, $\Av(\overline{a})(\varphi(x;b)) = f(tp(b/M))$. 
\end{definition}

\begin{remark} Every element $f$ in $\conv_{\mathbb{Q}}(\mathbb{F}_{M}^{\varphi})$ has many representative sequences. If $f \in \conv_{\mathbb{Q}}(\mathbb{F}_{M}^{\varphi})$, then $f = \sum_{i=1}^n r_iF_{c_i}^{\varphi}$ where each $c_i$ is in $M^{x}$, $r_i \in \mathbb{Q}^+$, and $\sum_{i=1}^n r_i = 1$. Let $m$ be the smallest number so that for every $i \leq n$, $m \cdot r_i \in \mathbb{N}$. Let
\begin{equation*}
    \mathbf{a}_i = \underbrace{c_i,...,c_i}_{m\cdot r_i \text{ times}}
\end{equation*}
Then, the concatenation of the $\mathbf{a}_i$'s is a representative sequence for $f$. 
\end{remark}

Finally, we end the introduction with an important fact about continuous functions on Stone spaces. This result is folklore and we provide a proof for the sake completeness. 

\begin{fact}\label{tdc} Let $S$ be a totally disconnected compact Hausdorff space. Then, a map $f:S \to [0,1]$ is continuous if and only if for every $\epsilon > 0$, there exists a collection of clopen sets $\mathcal{P} = \{C_1,...,C_m\}$ such that $\mathcal{P}$ forms a partition of $S$ and for each $i$ if $b,b' \in C_i$, then $|f(b) - f(b')| < \epsilon$.
\end{fact}

\begin{proof} First, we prove the forward direction. Assume that $f:S \to [0,1]$ is continuous. Let $B = \{B_{\epsilon_i}: i \leq n\}$ be a finite collection of open intervals of length $\epsilon$ which cover $[0,1]$. Then, $f^{-1}(B_{\epsilon_{i}}) = U_i$. Then,  $U_i= \bigcup_{j \in J_{i}} C_{i_j}$ where each $C_{i_j}$ is clopen. Now,
\begin{equation*}\bigcup_{i=1}^n f^{-1}(B_{\epsilon_i}) = \bigcup_{i =1}^n \bigcup_{j \in J_i} C_{i_j}
\end{equation*}is an open cover of $S$. So, $\bigcup_{k =1}^{m} C_k$ for some $k$'s in  $\{i_j: i \leq n, j \in J_i\}$ is a finite subcover. If $b,b' \in C_k$, then $f(b),f(b') \in B_{\epsilon_i}$ and so $|f(b)- f(b')|<\epsilon$. Choosing the atoms of the Boolean algebra generated by $\{C_k: k \leq m\}$ gives us a partition. 

Now, the other direction. We need to show that $f$ is continuous. Let $B_{\epsilon}$ be an open interval of length $\epsilon$. We want to show that $f^{-1}(B_{\epsilon})$ is an open set. If $f^{-1}(B_{\epsilon}) = \emptyset$, then we are finished. Assume that $p \in f^{-1}(B_{\epsilon})$. Notice that $f(p) \in B_{\epsilon}$. Choose $\delta$ such that $(f(p) - \delta, f(p) + \delta) \subset B_{\epsilon}$. Let $\mathcal{P} = \{C_1,...,C_m\}$ form a partition for $\delta$ and assume that $p \in C_p$. Then, $f(C_p) \subseteq (f(p) - \delta, f(p) + \delta)$ and so $C_p \subset f^{-1}(B_{\epsilon})$. Repeating this process for each $p \in f^{-1}(B_{\epsilon})$, we conclude that $f^{-1}(B_{\epsilon}) = \bigcup_{p \in B_{\epsilon}} C_p$. Therefore, $f^{-1}(B_{\epsilon})$ is open. 
\end{proof}

\section{Dependence}
The purpose of this section is to connect the notions of VC-dimension, NIP formulas, and sequential dependence. In particular, we show that if $\varphi(x;y)$ is NIP, then $\conv(\mathbb{F}_{M}^{\varphi})$, as a subset of $\mathbb{R}^{S_{y}(M)}$, is sequentially dependent. This result, Theorem \ref{sd}, follows implicitly by results in \cite{B}. This section gives a direct proof of this theorem. We begin by fixing some notation and recalling a quick fact about NIP formulas. 

\begin{definition} Let $\varphi(x;y)$ be a partitioned formula. Let $\mathcal{F}_\varphi$ be the family of definable subsets of $\mathcal{U}^x$ of the form $\{\varphi(x;b):b \in \mathcal{U}^y\}$. Likewise, we let $\mathcal{F}_{\varphi^*} = \{\varphi(a,y): a \in \mathcal{U}^x\}$.
\end{definition}

\subsection{VC Theory} We now continue by recalling some basic VC theory. The purpose of this subsection is to clearly state the VC theorem. 

\begin{definition}[VC-Dimension] Let $X$ be a set and let $\mathcal{F}$ be a family of subsets of $X$. Then, the \textit{VC-dimension of $\mathcal{F}$} is the largest $n$ such that there exists $A \subset X$, $|A| = n$, and,
\begin{equation*}
     \{K \subset A : \exists F \in \mathcal{F} \text{ so that } F \cap A = K\} = \mathcal{P}(A).
\end{equation*}
We denote the VC-dimension of $\mathcal{F}$ as $dim_{VC}(\mathcal{F})$. If no such $n$ exists, then we say $dim_{VC}(\mathcal{F}) = \infty$. 
\end{definition}

\begin{fact} If $\varphi(x;y)$ is NIP then $\mathcal{F}_{\varphi}$ and $\mathcal{F}_{\varphi^*}$ have finite VC-dimension. 
\end{fact}

\begin{definition}[$\epsilon$-Approximations] Let $X$ be a set, $\mathcal{F}$ be a family of subsets of $X$, and $\mu$ be a measure on $\mathcal{P}(X)$. Then, we say that a sequence of elements $a_1,...,a_n$ in $X$ is a \textit{$\epsilon$-approximation for $\mu$ over $\mathcal{F}$} if for every $F \in \mathcal{F}$, $|\mu(F) - \Av(\bar{a})(F)|<\epsilon$.
\end{definition}

Now, we may state the VC Theorem. 

\begin{theorem}[\cite{VC}] Let $X$ be a set and $\mathcal{F}$ a family of subsets of $X$. Assume that $\mathcal{F}$ has finite VC-dimension. Let $\mu$ be a probability measure concentrating on a finite subset of $X$, i.e. $\mu(A) = 1$ for some finite set $A$. Then, for every $\epsilon > 0$, there exists some constant $C_{\epsilon,d}$, depending only on $\epsilon$ and $d=dim_{VC}(\mathcal{F})$, and a sequence of points $c_1,...,c_m$ in $X$ such that $m \leq C_{d,\epsilon}$ and $c_1,...,c_m$ is an $\epsilon$-approximation of $\mu$ over $\mathcal{F}$. 
\end{theorem}

\begin{remark}\label{rem} We notice that if $\mathcal{F}$ is a family of subsets of $X$ with VC-dimension $d$, then for every $\epsilon > 0$ and for every probability measure $\mu$ concentrating on a finite set, there exists some $K_{d,\epsilon}$ depending only of $d$ and $\epsilon$ and a sequence of points $a_1,...,a_m$ in $X$ where $m = K_{d,\epsilon}$ such that $a_1,...,a_m$ is an $\epsilon$-approximation of $\mu$ over $\mathcal{F}$. In other words, for any measure concentrating on a finite set, we can find an $\epsilon$-approximation of \textbf{exactly} length $K_{d,\epsilon}$, e.g. we may take $K_{d,\epsilon} = C_{d,\epsilon}!$ since one can just concatenate sequences of length less than $C_{d,\epsilon}$ with themselves many times. 
\end{remark}

\subsection{Continuous VC Classes} We will explain dependence in the continuous context and explain how this relates to sequential dependence. 

\begin{definition}[Shattering] Let $X,Y$ be sets, $f: X \times Y \to [0,1]$, $r \in (0,1)$ and $\epsilon > 0$. Let $A \subset X$. We say that $f$ $(r,\epsilon)$-shatters $A$ if for every $K \subseteq A$, there exists some $b_K$ in $Y$ so that
\begin{equation*}
    \{a \in X: f(a,b_K) \leq r\} \cap A = K,
\end{equation*}
and,
\begin{equation*}
    \{a \in X: f(a,b_K) \geq r+ \epsilon\} \cap A = A \backslash K.
\end{equation*}
Moreover, we say that $B$ witnesses the $(r,\epsilon)$-shattering of $A$ if for each $K\subseteq A$, $B$ contains $b_K$ as above. 
\end{definition}

The following definition of dependence is equivalent to the one given in \cite{B}. 

\begin{definition}[Dependence]\label{dep} Let $X,Y$ be sets and let $f:X \times Y \to [0,1]$. We say that $f$ is $(r,\epsilon)$\textit{-independent} if every $n \in \mathbb{N}$, there exists $A \subset X$ where $|A| > n$ and $f$ $(r,\epsilon)$-shatters $A$. We say that $f$ is \textit{independent} if there exists some $r \in (0,1)$ and $\epsilon >0$ so that $f$ is $(r,\epsilon)$-independent. Finally, we say that $f$ is
\textit{dependent} if $\varphi$ is not independent. 
\end{definition}

\begin{fact}\label{D} If $f:X \times Y \to [0,1]$ is dependent, then $\mathbf{X} = \{f(a,y):a \in X\} \subset \mathbb{R}^{Y}$ is sequentially dependent.  
\end{fact}

\begin{proof} This follows directly from the definitions. 
\end{proof}

\subsection{Space of Dirac Measures} We prove the theorem advertised in the introduction of this section, namely Theorem \ref{sd}. We begin by discussing certain kinds of measures. 

We recall that if $X$ is a set and $a \in X$, then a Dirac measure $\delta_{a}$ on $\mathcal{P}(X)$ is a measure where for any $A$ in $\mathcal{P}(X)$, we have that $\delta_{a}(A) = 1$ if $a \in A$ and $\delta_{a}(A) = 0$ if $a \not \in A$. 

\begin{definition}[Convex Hull of Dirac Measures] Let $X$ be a set. Let $\delta(X)$ be the the collection of Dirac measures over $X$. Then, $\conv(\delta(X))$ is the convex hull of Dirac measures over $X$, i.e.
\begin{equation*}
\conv(\delta(X)) = \left\{\sum^n_{i=1}r_i\delta_{x_i}: n \in \mathbb{N},  x_i \in X,  r_i \in \mathbb{R}^+,  \sum_{i=1}^nr_i=1\right\}.
\end{equation*}
We will write $\conv(\delta(X))$ simply as $\conv(X)$. 
\end{definition}

\begin{definition} Let $f:X \times Y \to \{0,1\}$. Then, we define the map $\Psi_f: \conv(X) \times \conv(Y) \to [0,1]$ via
\begin{equation*}
    \Psi_f(\mu,\nu) = \int f(x,y)d\mu d\nu = \sum_{i=1}^n \sum_{j=1}^m r_i s_jf(a_i,b_j),
\end{equation*}
where $\mu = \sum_{i=1}^n r_i \delta_{a_i}$ and $\nu = \sum_{j=1}^m s_j \delta_{b_j}$. 
\end{definition}

\begin{theorem}\label{CVC} Assume that $\varphi(x;y)$ is an NIP formula. Consider $f: \mathcal{U}^x \times \mathcal{U}^y \to \{0,1\}$ via, 
\begin{equation*}
    f(a,b) =\begin{cases}
\begin{array}{cc}
1 & \mathcal{U}\models\varphi(a,b),\\
0 & otherwise.
\end{array}\end{cases}
\end{equation*} Then the map $\Psi_f: \conv(\mathcal{U}^x) \times \conv(\mathcal{U}^y) \to [0,1]$ is dependent. 
\end{theorem}
\begin{proof} Assume that the map $\Psi_f$ is independent. We will find a Boolean combination of $\varphi$ which is not NIP. Since $\Psi_f$ is independent, $\Psi_f$ is $(r_0,\epsilon_0)$-independent for some $r_0$ in $(0,1)$ and $\epsilon_0 > 0$. Let $r = r_0 + \frac{\epsilon_0}{3}$, $\epsilon = \frac{\epsilon_0}{3}$, and $\delta = \frac{\epsilon}{6} = \frac{\epsilon_0}{18}$. Then $\Psi_f$ is $(r,\epsilon)$-independent because $[r,r+\epsilon] \subset [r_0,r_0 + \epsilon_0]$. Now, $d_1 =dim_{VC}(\mathcal{F}_{\varphi})$ and $d_2 =dim_{VC}(\mathcal{F}_{\varphi^*})$ where $d_1,d_2 < \infty$ since $\varphi(x;y)$ is NIP. Let $n_{\star} = K_{\delta,d_1}$ and let $m_{\star} = K_{\delta,d_2}$ described in Remark \ref{rem}. Choose the smallest $w \in \{1,2,...,n_{\star} m_{\star}\}$ such that $\frac{w}{n_{\star}m_{\star}} \geq r$. Let $n_{\star} \times m_{\star} =\{(i,j): i \leq n_{\star}, j\leq m_{\star}\}$ and let $W = [n_{\star}\times m_{\star}]^w$, i.e. the subsets of $n_{\star}\times m_{\star}$ of size $|w|$. Now, consider the formula: 
\begin{equation*}
    \theta(x_1,...,x_{n_{\star}},y_1,...,y_{m_{\star}}) \equiv \neg \bigvee_{\alpha \in W} \bigwedge_{(i,j) \in \alpha} \varphi(x_i,y_j)
\end{equation*}
Notice that $\theta(\overline{x},\overline{y})$ takes in two sequences of elements and decides whether $\varphi(x;y)$ holds on a certain proportion of pairs of elements. In other words, if $a_1,...,a_{n_{\star}}$ is a sequence in $\mathcal{U}^{x}$ and $b_1,...,b_{m_{\star}}$ is a sequence in $\mathcal{U}^{y}$, then $\theta(\overline{a},\overline{b})$ determines whether $\Psi_{f}(\Av(\overline{a}),\Av(\overline{b}))$ is greater than $r$. 

We will now show that $\theta$ is not NIP. Assume that $\Psi_f$ $(r_0,\epsilon_0)$-shatters $A$. Then, for each $K \subseteq A$, there exists $\nu_K$ in $\conv(\mathcal{U}^{y})$ so that,
\begin{equation*}
    \{\mu \in \conv(\mathcal{U}^{x}): \Psi_f(\mu,\nu_K) \leq r\} \cap A = K,
\end{equation*}
and,
\begin{equation*}
    \{\mu \in \conv(\mathcal{U}^x): \Psi_F(\mu,\nu_K) \geq r+ \epsilon\} \cap A = A \backslash K.
\end{equation*}
 Then, for each $\mu \in A$, let $a_1^{\mu},...,a_{n_{\star}}^{\mu}$ be a $\delta$-approximation for that particular $\mu$ over $\mathcal{F}_{\varphi}$ of length exactly $n_{\star}$. Moreover, for each $K \subset A$, let $b^K_1,...,b^K_{m_{\star}}$ be a $\delta$-approximation for $\nu_K$ over $\mathcal{F}_{\varphi^*}$ of length exactly $m_{\star}$. Let $A_{\star} =\{(a_{1}^{\mu},...,a_{n_{\star}}^{\mu}): \mu \in A\} \subset \mathcal{U}^{l(x) \cdot n_{\star}}$ and $B_{\star} = \{(b_1^{K},...,b_{m_{\star}}^{K}): K \subset A\} \subset \mathcal{U}^{l(y)\cdot m_{\star}}$.  Then
\begin{equation*}
    |\Psi_f(\mu,\nu_K) - \Psi_f (\Av(\overline{a^{\mu}}),\Av(\overline{b^{K}}))| < 2\delta =  \frac{\epsilon}{3}.
\end{equation*}

Now, if $(a_1^{\mu},...,a_{n_{\star}}^{\mu})$ is in $A_{\star}$ and $(b_1^{K},...,b_{m_{\star}}^{K})$ is in $B_{\star}$, then by a standard computation we notice that,
\begin{equation*}
\mathcal{U} \models \theta(\overline{a^{\mu}},\overline{b^{K}}) \iff \Psi_{f}(Av(\overline{a^{\mu}}),Av(\overline{b^{K}})) \leq r  \iff  \Psi_{f}(\mu,\nu_{k}) \leq r_{0},
\end{equation*}
as well as, 
\begin{equation*}
\mathcal{U} \models \neg \theta(\overline{a^{\mu}},\overline{b^{K}}) \iff \Psi_{f}(Av(\overline{a^{\mu}}),Av(\overline{b^{K}})) \geq r + \epsilon  \iff  \Psi_{f}(\mu,\nu_{k}) \geq r_{0} + \epsilon. 
\end{equation*}
We conclude that $\theta$ is not NIP. However, by the fact stated in the introduction, $\varphi$ is also not NIP. Therefore, we have a contradiction. 
\end{proof}

Recall that $\mathbb{F}^{\varphi}_{M} = \{F_{a}^{\varphi} : a \in M\}$ as defined in Definition \ref{F}. 

\begin{corollary}\label{E} The map $\Eval: \conv(\mathbb{F}_{M}^{\varphi}) \times S_{y}(M) \to [0,1]$ where $\Eval(f,p) = f(p)$ is dependent. 
\end{corollary}
\begin{proof}
Assume not. Then $\Eval$ is $(r,\epsilon)$-independent for some $r$ and $\epsilon$. Assume that $A \subseteq \conv(\mathbb{F}_{M}^{\varphi})$ is $(r,\epsilon)$-shattered. Then, for each subset $K$ of  $A$, there exists $p_K$ in $S_{y}(M)$ where $\{f \in \conv(\mathbb{F}_{M}^{\varphi}): f(p_K) \leq r\} \cap A = K$ and $\{f \in \conv(\mathbb{F}_{M}^{\varphi}): f(p_K) \geq r + \epsilon \} \cap A = A \backslash K$. 
Since $f$ is in $\conv(\mathbb{F}_{M}^{\varphi})$, we note that $f = \sum_{i=1}^n r_iF_{a_i}^{\varphi}$ for some $a_i$ in $M^{x}$ and $r_i > 0$. Let $\mu_f = \sum_{i=1}^n r_i \delta_{a_i}$. For each $K \subset A$, we let $b_K \in \mathcal{U}^{y}$ be a realization of $p_K$. Let $A_{\star} = \{\mu_f: f \in A\} \subset \conv(\mathcal{U}^x)$. Then, for any $\mu_{f} \in A_{\star}$, we have,
\begin{equation*}
\Psi_f(\mu_f,\delta_{b_K}) = \Eval(f,tp(b_K/M)) = \Eval(f,p_K). 
\end{equation*} 
Therefore if $\Eval$ is independent then so is the map $\Phi_{f}$. Since $\varphi(x;y)$ is NIP, this contradicts Theorem \ref{CVC}. 
\end{proof}

\begin{theorem}\label{sd} Assume that $\varphi(x;y)$ is an NIP formula. Then $\conv(\mathbb{F}_{M}^{\varphi}) \subseteq \mathbb{R}^{S_{y}(M)}$ is sequentially dependent.  
\end{theorem}

\begin{proof} This follows directly from Corollary \ref{E} and Fact \ref{D}. 
\end{proof}

\section{Local Keisler Measures} 
We recall that $\mathcal{U}$ is a monster model of $T$, $M$ is a small elementary substructure of $\mathcal{U}$, and $\varphi(x;y)$ is a partitioned $L$-formula. We do \textbf{not} require $\varphi(x;y)$ to be NIP unless explicitly stated otherwise. In the first subsection, we define the weak notion of \textit{$\varphi$-generic stability} (Definition \ref{definitions}). Assuming that $\varphi(x;y)$ is NIP, we will show that if $\mu$ is in $\mathfrak{M}_{\varphi}(\mathcal{U})$ and $\mu$ is $\varphi$-generically stable over some $M$ where $M$ is countable, then for every $\epsilon > 0$, there exists a sequence of points $a_1,...,a_n$ of elements in $M^x$ such that this sequance is an $\epsilon$-approximation for $\mu$ over $\mathcal{F}_{\varphi}$ (Theorem \ref{mt1}). In the second subsection, we will then use Theorem \ref{mt1} to prove our main result described in the introduction. Some of the proofs in this section are standard but are provided for the sake of completeness. We begin with some definitions. 

\subsection{Basic properties of Local Keisler Measures} 

\begin{definition}Let $\mu \in \mathfrak{M}_{\varphi}(\mathcal{U})$. 
\begin{enumerate}[($i$)]
\item $\mu$ is $M$\textit{-invariant} if for every $\varphi$-formula $\psi(x;b)$ and every automorphism $\sigma \in \Aut(\mathcal{U}/M)$, we have $\mu(\psi(x;b)) = \mu(\psi(x;\sigma(b)))$. 
\item $\mu$ is $\varphi$\textit{-invariant over $M$} if for every $b,c \in \mathcal{U}^y$ such that $tp(b/M) = tp(c/M)$, we have that $\mu(\varphi(x;b)) = \mu(\varphi(x;c))$. 
\end{enumerate} 
\end{definition} 

Since $M$ is assumed to be small, the collection of measures which are $\varphi$-invariant over $M$ contains the collection of $M$-invariant measures. Notice that $\varphi$-invariance only mentions instances of $\varphi$ and does not mention Boolean combinations of $\varphi$. Our next definition connects $\varphi$-invariant measures over $M$ to functions from $S_{y}(M)$ to $\mathbb{R}$. 

\begin{definition}\label{functions} Assume that $\mu \in \mathfrak{M}_{\varphi}(\mathcal{U})$ and $\mu$ is $\varphi$-invariant over $M$. Then, we define the function $F^{\varphi}_{\mu}:S_{y}(M) \to [0,1]$ via $F_{\mu}^{\varphi}(p) = \mu(\varphi(x;b))$ where $b \models p$. 
\end{definition}

Definition \ref{functions} allows us to transfer problems involving finitely additive measures to questions involving functions. We will soon see that special kinds of measures correspond to special kinds of functions. Let us now describe these special measures.    

\begin{definition}[Properties of Measures]\label{definitions} Let $\mu \in \mathfrak{M}_{\varphi}(\mathcal{U})$.
\begin{enumerate}[($i$)]
    \item We say that $\mu$ is \textit{$\varphi$-definable over $M$} if for every $\epsilon > 0$ there exist $L$-formulas $\psi_1(y),...,\psi_n(y)$ with parameters only from $M$, such that:
\begin{enumerate}
    \item The collection $\{\psi_i(y): i \leq m\}$ forms a partition of $\mathcal{U}^{y}$. 
    \item If $\mathcal{U} \models \psi_i(e)\wedge \psi_i(c)$, then $|\mu(\varphi(x;e)) - \mu(\varphi(x;c))| < \epsilon$. 
\end{enumerate}
\item We say that $\mu$ is \textit{finitely satisfiable over $M$} if for every $\varphi$-formula $\psi(x)$ such that $\mu(\psi(x)) > 0$, there exists some $a$ in $M^{x}$ such that $\mathcal{U} \models \psi(a)$. 
\item We say that $\mu$ is \textit{ $\varphi$-generically stable over $M$} if $\mu$ is both $\varphi$-definable and finitely satisfiable over $M$. 
\end{enumerate}
\end{definition}

We now connect the kinds of measures defined above with functions from $S_{y}(M)$ to $\mathbb{R}$.

\begin{proposition}\label{disc} Let $\mu \in \mathfrak{M}_{\varphi}(\mathcal{U})$. Then $\mu$ is $\varphi$-definable over $M$ if and only if $\mu$ is $\varphi$-invariant and the map $F^{\varphi}_{\mu}:S_{y}(M) \to [0,1]$ is continuous. 
\end{proposition}

\begin{proof} This follows from Fact \ref{tdc} and the assumption that $M$ is small. 
\end{proof}

\begin{proposition}\label{fswd} If $\mu \in \mathfrak{M}_{\varphi}(\mathcal{U})$ and $\mu$ is finitely satisfiable over $M$, then $\mu$ is $M$-invariant, and in partiuclar, $\mu$ is $\varphi$-invariant. 
\end{proposition}

\begin{proof} The proof is identical to the standard Keisler measure statement \cite{guide}, but we provide it for the sake of clarity. Assume that $\psi(x;b)$ is a $\varphi$-formula, $\sigma \in \Aut(\mathcal{U}/M)$ and $\mu(\psi(x;b)) > \mu(\psi(x;\sigma(b)))$. Then, $\mu(\psi(x;b) \wedge  \neg \psi(x;\sigma(b)) > 0$. Since $\mu$ is finitely satisfiable over $M$, there exists some $a$ in $M^{x}$ such that $\mathcal{U} \models \psi(a,b) \wedge \neg \psi(a,\sigma(b))$. Then, $tp(b/M) \neq tp(\sigma(b)/M)$ and we have a contradiction. 
\end{proof}

Recall that if $A \subset \mathbb{R}^X$, then $cl(A)$ is the topological closure of $A$ in $\mathbb{R}^X$. We will now connect this closure property with finite satisfiability. 

\begin{proposition}\label{fslp} Let $\mu \in \mathfrak{M}_{\varphi}(\mathcal{U})$ and assume that $\mu$ is finitely satisfiable over $M$. Let $X = S_{y}(M)$. Then $F_{\mu}^{\varphi} \in cl(\conv_{\mathbb{Q}}(\mathbb{F}^{\varphi}_M)) \subset \mathbb{R}^X$. 
\end{proposition}

\begin{proof}
By Fact \ref{fswd}, the map $F_{\mu}^{\varphi}$ is well defined. Let $U$ be some open subset of $\mathbb{R}^{X}$ containing $F^{\varphi}_\mu$. Without loss of generality,
\begin{equation*}U = \bigcap_{i=1}^n \{f \in \mathbb{R}^X: r_i < f(q_i) < s_i\}.
\end{equation*}

Choose some $b_i \models q_i$ for each $q_i$. Then, we have that $r_i<\mu(\varphi(x;b_i))<s_i$. We notice that $\varphi(x;b_1),...,\varphi(x;b_n)$ generate a finite Boolean algebra. Let $\theta_1(x),...,\theta_m(x)$ be the atoms of this Boolean algebra. If $\mu(\theta_i(x)) >0$ then by finite satisfiability we can find $a_i$ in $M^{x}$ such that $\mathcal{U} \models \theta_i(a_i)$. Define

\begin{equation*}
f(y) = \sum_{\{i: \mu(\theta_i(x)) > 0\}} \mu(\theta_i(x))F^{\varphi}_{a_i}(y).
\end{equation*} 

\noindent Then, $f \in \conv(\mathbb{F}_{M}^{\varphi}) \cap U$. Since $\conv_{\mathbb{Q}}(\mathbb{F}^{\varphi}_{M})$ is a dense subset of $\conv(\mathbb{F}^{\varphi}_{M})$ with the induced topology from $\mathbb{R}^X$, there exists some $g \in \conv_{\mathbb{Q}}(\mathbb{F}^{\varphi}_{M}) \cap U$. 
\end{proof}

\begin{lemma}\label{ML} Assume that $\varphi(x;y)$ is NIP. Let $\mu \in \mathfrak{M}_{\varphi}(\mathcal{U})$. Assume that $\mu$ is finitely satisfiable over $M$ and $|M|=\aleph_0$. Then there exists some sequence $(f_i)_{i \in \mathbb{N}}$ of elements in $\conv_{\mathbb{Q}}(\mathbb{F}^{\varphi}_M)$ such that $f_i \to F^{\varphi}_{\mu}$. 
\end{lemma}

\begin{proof} By Proposition \ref{fslp}, $F_{\mu}^{\varphi}$ is in $cl(\conv_{\mathbb{Q}}(F_{M}^{\varphi}))$. Since $\varphi(x;y)$ is NIP, Theorem \ref{sd} holds. Since $|\conv_{\mathbb{Q}}(\mathbb{F}_{M}^{\varphi})| = \aleph_0$, statement $(i)$ of Theorem \ref{equiv} is satisfied. Therefore statement $(ii)$ of Theorem \ref{equiv} holds, so there exists a sequence of elements $f_i \in \conv(\mathbb{F}_{M}^{\varphi})$ such that $f_i \to F_{\mu}^{\varphi}$. 
\end{proof}

We now prove the main theorem of this section. 

\begin{theorem}\label{mt1} Assume that $\varphi(x;y)$ is NIP. Let $\mu \in \mathfrak{M}_{\varphi}(\mathcal{U})$. Assume that $\mu$ is $\varphi$-generically stable over $M$ and $|M| = \aleph_0$. Then for every $\epsilon > 0$, there exists a sequence  $a_1,...,a_n$ of elements in $M^{x}$ such that, 
\begin{equation*}
    |\mu(\varphi(x;b)) - \Av(\overline{a})(\varphi(x;b))| < \epsilon.
\end{equation*}
\end{theorem}

\begin{proof} Fix $\epsilon$. By Lemma \ref{ML}, there exists a sequence $(f_i)_{i \in \mathbb{N}}$ of elements in $\conv_{\mathbb{Q}}(\mathbb{F}_{M}^{\varphi})$ such that $f_i \to F_{\mu}^{\varphi}$. By Fact \ref{disc}, we know that $F_{\mu}^{\varphi}$ is continuous. Since each $f_i$ is also continuous, we may apply Theorem \ref{K}. Therefore, $f_i \to_{w} F_{\mu}^{\varphi}$. By Mazur's Lemma,
there exists a sequence $g_i \in \conv_{\mathbb{Q}}(\mathbb{F}_{M}^{\varphi})$ such that $g_i \to_u F_{\mu}^{\varphi}$. Choose $g_n$ such that
\begin{equation*}
    \sup_{q \in S_{y}(M)}|F_{\mu}^\varphi(q) - g_n(q)| < \epsilon.
\end{equation*}
Notice that $g_n = \sum_{i=1}^n r_iF^{\varphi}_{c_i}$ and let $a_1,...,a_n$ be a representative sequence for $g_n$. Then for each $b \in \mathcal{U}^{y}$, 
\begin{equation*}
    |\mu(\varphi(x;b)) - \Av(\overline{a})(\varphi(x;b))| = |F_{\mu}^{\varphi}(tp(b/M)) - g_n(tp(b/M))|.
\end{equation*}
Since $M$ is small, we may conclude that
\begin{equation*}
    \sup_{b \in \mathcal{U}}|\mu(\varphi(x;b)) - \Av(\bar{a})(\varphi(x;b))| 
    = \sup_{q \in S_{y}(M)}|F_{\mu}^\varphi(q) - g_n(q)|<\epsilon. 
\end{equation*}
\end{proof} 
\subsection{Main Result}
The purpose of this section is to prove our Main Theorem as described in the introduction. We again do \textbf{not} require $\varphi(x;y)$ to be NIP in this section unless explicitly stated.
We will see that if we strengthen our definition of \textit{$\varphi$-definability} and in turn our definition of \textit{$\varphi$-generic stability}, we can prove our main result. We begin this section by considering different families of $\varphi$-definable sets.

\begin{definition} For a fixed partitioned $L$-formula $\varphi(x;y)$, we define $\Delta_{\varphi}$ as the Boolean algebra of partitioned formulas generated by $\{\varphi(x;y_i): i \in \mathbb{N}, l(y_i) = l(y)\}$. If $\theta(x;\overline{y})$ is an element of $\Delta_{\varphi}$, we let $\Def_{\theta}(\mathcal{U})$ be the Boolean subalgebra of $\Def_{x}(\mathcal{U})$ generated by $\{\theta(x;\overline{b}): \overline{b} \in \mathcal{U}^{\overline{y}}\}$. Moreover, we let $I_{\theta}$ denote the obvious restriction map from $\mathfrak{M}_{\varphi}(\mathcal{U})$ to $\mathfrak{M}_{\theta}(\mathcal{U})$. For notational purposes, if $\mu \in \mathfrak{M}_{\varphi}(\mathcal{U})$ and $\theta(x;\overline{y})$ is in $\Delta_{\varphi}$, then we write $I_{\theta}(\mu)$ simply as $\mu_{\theta}$.
\end{definition}

For example, the formulas $\varphi(x;y_1) \wedge \varphi(x;y_2)$, $\varphi(x;y_1) \triangle \varphi(x;y_2)$, and $\bigvee_{i=1}^{14}\varphi(x;y_i)$ are all elements of $\Delta_{\varphi}$. 

We now give the definition generic stability and finite approximability as well as some relations between the properties we have already defined. 

\begin{definition}\label{cont} Let $\mu \in \mathfrak{M}_{\varphi}(\mathcal{U})$. Then $\mu$ is \textit{definable over $M$} if for every $\theta(x;\overline{y})$ in $\Delta_{\varphi}$ and $\epsilon >0$ there exist $L$-formulas $\rho_1(\overline{y})$,...,$\rho_m(\overline{y})$ with parameters only from $M$, such that: 
\begin{enumerate}[($i$)] 
\item The collection $\{\rho_i(\overline{y}): i \leq m\}$ forms a partition of $\mathcal{U}^{\overline{y}}$. 
\item if $\mathcal{U} \models \rho_i(\overline{e}) \wedge \rho_i(\overline{c})$, then $|\mu(\theta(x;\overline{e})) - \mu(\theta(x;\overline{c}))|<\epsilon$.
\end{enumerate} 
We say that $\mu$ is \textit{generically stable over $M$} if $\mu$ is definable and finitely satisfiable over $M$. 
\end{definition}

\begin{definition}\label{finap} Let $\mu \in \mathfrak{M}_{\varphi}(\mathcal{U})$. Then, $\mu$ is \textit{finitely approximable} if for every $\theta(x;\overline{y})$ in $\Delta_{\varphi}$ and for every $\epsilon > 0$, there exists a sequence $a_1,...,a_n$ in $M^{x}$ such that for any $b \in \mathcal{U}^{\overline{y}}$
    \begin{equation*}
        |\mu(\theta(x;\overline{b})) - \Av(\overline{a})(\theta(x;\overline{b}))| < \epsilon. 
    \end{equation*}
\end{definition}

\begin{proposition}[Basic Properties]\label{equiv2} Let $\mu \in \mathfrak{M}_{\varphi}(\mathcal{U})$. 
\begin{enumerate}[($i$)]
    \item  $\mu$ is $M$-invariant if and only if for every $\theta(x;\overline{y})$ in $\Delta_{\varphi}$, $\mu_{\theta}$ is $\theta$-invariant over $M$. 
    \item $\mu$ is definable over $M$ if and only if for every $\theta(x;\overline{y})$ in $\Delta_{\varphi}$, the measure $\mu_{\theta} \in \mathfrak{M}_{\theta}(\mathcal{U})$ is $\theta$-definable over $M$. 
    \item $\mu$ is definable over $M$ if and only if $\mu$ is $M$-invariant and for every $\theta(x;\overline{y})$ in $\Delta_{\varphi}$ the map $F_{\mu_{\theta}}^{\theta}:S_{\overline{y}}(M) \to [0,1]$ is continuous. 
    \item $\mu$ is generically stable over $M$ if and only if for each $\theta(x;\overline{y}) \in \Delta_{\varphi}$, the measure $\mu_{\theta} \in \mathfrak{M}_{\theta}(\mathcal{U})$ is generically stable over $M$.
    \item If $\mu$ is finitely approximable over $M$, then $\mu$ is generically stable over $M$. 
\end{enumerate}
\end{proposition}

\begin{proof} $(ii)$ and $(iv)$ follow directly from the definitions. $(i)$ follows from the fact that $M$ is small. $(iii)$ follows from Proposition \ref{disc} and $(ii)$. The proof of $(v)$ is identical to the general Keisler measure case \cite{guide}. However, for clarity, we present the proof here. Assume that $\mu$ is finitely approximable over $M$. We first show that $\mu$ is finitely satisfiable over $M$. Assume that $\mu(\theta(x;\overline{b}))> \delta > 0$. Then, there exists some $a_1,...,a_n$ in $M^{x}$ such that $|\mu(\theta(x;\overline{b})) - \Av(\overline{a})(\theta(x;\overline{b}))| < \frac{\delta}{2}$. Therefore, $\Av(\overline{a})(\theta(x;\overline{b})) > 0$ and so for some $a_i$ in our sequence, we have that $\mathcal{U} \models \theta(a_i, \overline{b})$. Now we show that $\mu$ is definable over $M$. It is clear that for every $\theta(x;\overline{y})$ in  $\Delta_{\varphi}$ the map $F^{\theta}_{\mu_{\theta}}$ is well-defined and a uniform limit of continuous functions. Therefore, $F^{\theta}_{\mu_{\theta}}$ is continuous. By $(iii)$, we may conclude that $\mu$ is definable over $M$ and hence generically stable over $M$. 
\end{proof}

We now present some properties for measures which are well known for types \cite{guide}. These properties will allow us to reduce our main result to the countable case and so we may apply Theorem \ref{mt1}. 

\begin{proposition}\label{fsi} Let $\mu \in \mathfrak{M}_{\varphi}(\mathcal{U})$. Assume that $\mu$ is finitely satisfiable over $N$ and $\mu$ is $M$-invariant for $M \prec N$. Then, $\mu$ is finitely satisfiable over $M$. 
\end{proposition}

\begin{proof} 
Let $\psi(x;\overline{b})$ be a $\varphi$-formula and assume that $\mu(\psi(x;\overline{b})) >0$. Let $N_1$ realize a coheir of $tp(N/M)$ over $M\overline{b}$. Let $\sigma \in \Aut(\mathcal{U}/M)$ such that $\sigma:N \to N_1$. Then, $\mu(\psi(x;\sigma^{-1}(\overline{b}))) = \mu(\psi(x;\overline{b}))$ because $\mu$ is $M-$invariant. By finite satisfiability over $N$, there exists $a \in N$ such that $\mathcal{U} \models \psi(a,\sigma^{-1}(\overline{b}))$. Then, $\mathcal{U} \models \psi(\sigma(a),\overline{b})$ where $\sigma(a) \in N_1$.  Therefore, $N_1\cap \{a \in \mathcal{U}: \mathcal{U} \models \psi(a;\overline{b})\}$ is non-empty. By the coheir hypothesis, $M \cap \{ a \in \mathcal{U} : \mathcal{U} \models \psi(a;\overline{b})\}$ is non-empty. Therefore, $\mu$ is finitely satisfiable over $M$.
\end{proof} 

\begin{proposition}\label{count} Let $\mu \in \mathfrak{M}_{\varphi}(\mathcal{U})$. If $\mu$ is generically stable over $M$, then there exists $M_0 \prec M$ such that $|M_0| = \aleph_0$ and $\mu$ is generically stable over $M_0$. 
\end{proposition}

\begin{proof} Assume that $\mu$ is generically stable over $M$. Then, for each $\theta(x;\overline{y})$ in $\Delta_{\varphi}$ and for each $\epsilon_n = \frac{1}{n}$, we can find $\mathcal{P}_{n}^\theta=\{\psi_1(\overline{y}),...,\psi_n(\overline{y})\}$ a partition of $\mathcal{U}^{\overline{y}}$ as in Definition \ref{cont}. Let $C$ be the collection of parameters used in each formula from each partition. Find $M_0 \prec M$ such that $C \subset M_0$ and $|M_0| = \aleph_0$. Then $\mu$ is definable over $M_0$, and in particular, $M_0$-invariant. By Proposition \ref{fsi}, we conclude that $\mu$ is finitely satisfiable over $M_0$. 
\end{proof}

We now essentially reduce our main result to Theorem \ref{mt1}. 

\begin{lemma}\label{lem} Assume that $\varphi(x;y)$ is NIP. Let $\mu \in \mathfrak{M}_{\varphi}(\mathcal{U})$. If $\mu$ is generically stable over $M$, then there exists a sequence $a_1,...,a_n$ in $M^{x}$ such that for any $b \in \mathcal{U}^y$, 
\begin{equation*}
    |\mu(\varphi(x;b)) - \Av(\overline{a})(\varphi(x;b))| < \epsilon. 
\end{equation*}
\end{lemma}

\begin{proof} By Theorem \ref{count}, $\mu$ is generically stable over some $M_0$ where $M_0 \prec M$ and $|M_0| = \aleph_0$. Then, we may apply Theorem \ref{mt1} to $\mu$ and $M_0$. Since $M_0 \subseteq M$, we are done. 
\end{proof}

\begin{theorem}\label{last} Assume that $\varphi(x;y)$ is NIP. Let $\mu \in \mathfrak{M}_{\varphi}(\mathcal{U})$. Then $\mu$ is generically stable over $M$ if and only if $\mu$ is finitely approximable over $M$. 
\end{theorem}

\begin{proof} By $(v)$ of Proposition \ref{equiv2}, if $\mu$ is finitely approximable over $M$, then $\mu$ is generically stable over $M$. We only need to show the other direction. Assume $\mu$ is generically stable over $M$. By $(iv)$ of Proposition \ref{equiv2}, for any $\theta(x;\overline{y})$ in $\Delta_{\varphi}$, $\mu_{\theta}$ is generically stable over $M$. By construction, $\mu(\theta(x;\overline{b})) = \mu_{\theta}(\theta(x;\overline{b}))$ for every $\overline{b} \in \mathcal{U}^{\overline{y}}$. By Lemma \ref{lem}, for every $\epsilon > 0$ there are $a_1,...,a_n$ in $M^{x}$ so that for every $\overline{b} \in \mathcal{U}^{\overline{y}}$,  
\begin{equation*}
    |\mu_{\theta}(\theta(x;\overline{b})) - \Av(\overline{a})(\theta(x;\overline{b}))| <\epsilon.
\end{equation*}
Since $\mu(\theta(x;b)) = \mu_{\theta}(\theta(x;b))$, we conclude that $\mu$ is finitely approximable over $M$. 
\end{proof}

\end{document}